\documentclass[11pt,reqno]{article}
\usepackage{amsfonts,amssymb,verbatim,amsmath, amsthm}
\usepackage{pdfsync}
\usepackage{graphicx}
\usepackage{color}
\usepackage{tikz}
\usepackage{pgfplots}
\pgfplotsset{compat=newest, ticks=none}
\usepackage{fancyhdr} 
\numberwithin{equation}{section}

\newtheorem{theorem}{Theorem}[section]

\newtheorem{lemma}[theorem]{Lemma}
\newtheorem{proposition}[theorem]{Proposition}

\usepackage{authblk}
\AtBeginDocument{}
\allowdisplaybreaks

\newcommand{\bpr}{\begin{proof}\hspace{3pt}}
\newcommand{\epr}{\end{proof}}
\newcommand{\lb}{\left(}
\newcommand{\rb}{\right)}

\renewcommand{\Re}{\operatorname{Re}}

\def\R{{\mathbb{R}}}

\title{Unique continuation for a non bi-Laplacian fourth order elliptic operator}
\author[1]{A.~Ghosh \thanks{ghosh@math.cas.cz}}
\affil[1]{Institute of Mathematics, CAS, Czech Republic.}
\author[1]{T.~Ghosh \thanks{iasghosh@ust.hk}}
\affil[2]{Institute for Advanced Study, The Hong Kong University of Science and Technology, Hong Kong.}

\date{}
\begin{document}

\maketitle
\begin{abstract}
This paper discusses the unique continuation principal of the solutions of the  following perturbed fourth order elliptic differential operator $\mathcal{L}_{A,q}u=0$, where
\[
\mathcal{L}_{A,q}(x,D)\ =\ \sum_{j=1}^nD^4_{x_j} + \sum_{j=1}^n A_jD_{x_j} + q, \qquad (A, q) \in  W^{1,\infty}(\Omega,\mathbb{C}^n) \times L^{\infty}(\Omega,\mathbb{C}) \]
whose principal term is not given by some integer power of the Laplacian operator. We derive some suitable Carleman estimates which is the main tool to prove the unique continuation principle. As a by-product, we also deduce some stability estimate and prove the strong unique continuation principle in $2$-dimension. 
\end{abstract}

\section{Introduction}
\label{S1}
\setcounter{equation}{0}
Let $\Omega\subset\mathbb{R}^n$, $n\geq 2$ be a bounded connected open set. Let us consider the following fourth order operator
\begin{equation}
\label{operator}
\mathcal{L}_{A,q}(x,D)= \sum_{j=1}^nD^4_{x_j} + \sum_{j=1}^n A_jD_{x_j} + q
\end{equation}
where $A=(A_j)_j \in W^{1,\infty}(\Omega,\mathbb{C}^{n}), q\in L^{\infty}(\Omega,\mathbb{C})$ and $ D = \frac{1}{i}\nabla$. Throughout the paper we assume this regularity on $A$ and $q$. The operator $\mathcal{L}_{A,q}(x,D)$ is a positive definite elliptic operator with the principal part $\sum_{j=1}^n D^4_{x_j}$ which is self-adjoint on $L^2(\Omega)$. 
The purpose of this article is to discuss the unique continuation principle (UCP) of the solutions $u$ of such elliptic fourth order partial differential operator $\mathcal{L}_{A,q}u=0$. Ideally, this principle asserts that any solution of an elliptic equation that vanishes in a small ball, must be identically zero on the whole domain. It can also be interpreted as, given two regions $\Omega_1\subset\Omega_2$, a solution $u$ to $\mathcal{L}_{A,q}u = 0$ is uniquely determined on the larger set $\Omega_2$ by its values on the smaller set $\Omega_1$. The earliest such result for real analytic  coefficients is known as \textit{Holmgren's uniqueness theorem}, see \cite{JohnF}. The scalar second order case is well understood, we mention here the seminal articles \cite{CAR,AKS}, and the expository text \cite{KHTD} and reference therein as well.
In general, the corresponding theory for elliptic equations of order greater than two is much less discussed. Qualitatively, the case of higher order operators is different from the second order operators. We cite \cite{alinhac} in this regard and will get back it in more details at the end of this discussion. Higher order elliptic equations are common in the study of continuum mechanics, in the related field of elasticity, and application in engineering design as well, see \cite{Campos,PBook}. We mention the works  \cite{ARV,LCL,Protter,LeBP} where the UCP for some integer ($\geq 2$) power of Laplacian operator has been discussed. Here in \cite{koch}, we find the discussion of the unique continuation of the product of elliptic operators.  In comparison to the classical bi-Laplacian operator $(-\Delta)^2 = (\sum_{j=1}^nD^2_{x_j})^2$ say, the principal part of our operator $\sum_{j=1}^nD^4_{x_j}\, (\neq (-\Delta)^2)$ does not involve the mixed derivative terms $D^4_{x_ix_j}$, $i\neq j$. Thus, our operator can not be viewed as a higher order iteration of some second order elliptic operator. Moreover, in general it can not be written as the product of two elliptic operators, except in 2-dimension. This encourages us to make a fresh study of the UCP for this operator $\mathcal{L}_{A,q}(x,D)$. UCP results are often regarded as a tool to solve certain problems in solvability
of the related adjoint problem, inverse problems and control theory, see for instance \cite{TataruD,CFZC,LRLG}. Earlier, the second author has considered this operator to study the inverse boundary value problem of recovering the coefficients $A,q$ from the associated boundary Cauchy data, see \cite{GT}. Similar inverse boundary value problems for perturbed bi-harmonic and poly-harmonic operator has been discussed in \cite{KRU1,KRU2,VT,BG}. 

Now we announce the results obtained in this work. We prove quite a few theorems. Our first set of results consists of the so-called weak UCP (WUCP) and the UCP for the local Cauchy data.

\begin{theorem}[(WUCP)]
\label{T1}
Let $u\in H^4(\Omega)$ satisfies
\[ \mathcal{L}_{A,q}\,u =0 \quad \mbox{ in } \ \Omega.\]
Also let $\omega\subset\Omega$ be a non-empty open subset such that 
\[ u  = 0 \quad \mbox{ in } \ \omega, \]
then $u = 0$ in $\Omega$.
\end{theorem}
As an application of the above result, we deduce the UCP for local Cauchy data.

\begin{theorem}[(UCP for local Cauchy data)]
\label{T2}
Let $\Omega\subset\mathbb{R}^n$ have
smooth boundary, and let $\Gamma$ be a non-empty open subset of $\partial\Omega$. If $u\in H^4(\Omega)$ satisfies
\begin{equation*}
\begin{aligned}
\mathcal{L}_{A,q}\,u &=0 \quad \mbox{ in } \ \Omega,\\
\Big( u, \partial_\nu u\Big)_{\Gamma\times\Gamma} &= \Big( \partial^2_\nu u, \partial^3_\nu u\Big)_{\Gamma\times\Gamma}  = 0,
\end{aligned}
\end{equation*}
then $u = 0$ in $\Omega$.
\end{theorem}

There are various approaches to obtain UCP for elliptic equations, at least for the second order elliptic equations. In general such methods consist of either Carleman type estimates (\cite{Hor3,Hor4,KRS,WTH,KHTD}) or, Almgren's frequency function method (\cite{GSMN,GNLF,ARRV}).
In this paper, we rely on developing a class of Carleman estimates as our main tool and apply it in certain ways to establish the weak UCP and stability estimate. Here we mention few expository notes \cite{Lerner,SALON,TDN} which turns out to be very useful to carry out our work. 
  
We would like to emphasize here few essential contrast between our leading operator $\sum_{j=1}^n D^4_{x_j} $ and the bi-Laplacian operator $(-\Delta)^2 $. Let $\varrho\in\mathbb{R}^n$ be a non-zero vector; Then we prove that the following Carleman estimate (cf. Proposition \ref{P2})
\begin{equation}
\label{carF}
\|e^\frac{\varrho\cdot x}{h}h^4\sum_{j=1}^nD_{x_j}^4(e^{-\frac{\varrho\cdot x}{h}}w)\|_{L^2(\Omega)} \gtrsim h\left(\|w\|^2_{L^2(\Omega)} + \|h\nabla w\|^2_{L^2(\Omega)}\right)^{\frac{1}{2}} \gtrsim h\|w\|_{L^2(\Omega)} 
\end{equation}
holds for all $w\in C_c^{\infty}(\Omega)$ and $0<h<1$ small enough.  However, if the principal part is a bi-Laplacian $(-\Delta)^2$ operator, then we could have 
the following Carleman estimate \cite{KRU1}: 
\begin{equation}
\label{carB}
\|e^\frac{\varrho\cdot x}{h}(-h^2\Delta)^2(e^{-\frac{\varrho\cdot x}{h}}w)\|_{L^2(\Omega)} \gtrsim h\|e^\frac{\varrho\cdot x}{h}(-h^2\Delta)(e^{-\frac{\varrho\cdot x}{h}}w)\|_{L^2(\Omega)} \gtrsim h^2\|w\|_{L^2(\Omega)} 
\end{equation}
which holds for all $w\in C_c^{\infty}(\Omega)$ and $0<h<1$ small enough.
Notice that \eqref{carF} offers better lower-estimate compared to \eqref{carB} as $0<h<1$ which is due to the structure of the principal part of the respective operators (as proof indicates in Section \ref{S2}).

Also we would like to emphasize that though the Carleman estimate (\ref{carF}) is an interior estimate, the estimate up to the boundary can be derived from it (cf. proof of Theorem \ref{T4}) using the lift of the trace operator. Furthermore, a different type of boundary Carleman estimate has been proved in \cite[Theorem 3.1]{GT}.
Here is our next result.

For any smooth function $\varphi$, let us define
\begin{equation*}
\Omega_\delta := \Omega \cap \{\varphi >\delta\} \quad \text{ and } \quad \partial\Omega_\delta := \partial\Omega\cap \{\varphi >\delta\}.
\end{equation*}

\begin{theorem}[(Stability estimate)]
\label{T4}
Let $\varphi$ be any function which satisfies the \eqref{new5}, and $\partial\Omega_0\subset \Gamma $ where $\Gamma\subseteq \partial\Omega$. 
Suppose that $u\in H^4(\Omega)$ solve the Cauchy problem
\begin{equation*}
\begin{cases}
\mathcal{L}_{A,q}\,u = f \quad \text{ in } \ \Omega,\\
\partial^k_{\nu} u = g^k \quad \ \ \text{ on } \ \Gamma, \quad k=0,...,3,
\end{cases}
\end{equation*}
with $f\in L^2(\Omega)$ and $g^k\in H^{\frac{7}{2}-k}(\Gamma)$. Then there exists constant $C>0$, depending on $\delta, \Omega, \Gamma, \|A\|_{W^{1,\infty}(\Omega)}$, $\|q\|_{L^\infty(\Omega)}$, $n$ only and $\theta\in (0,1)$, depending on $\delta$, such that we have,
\begin{equation}
\label{SE}
\|u\|_{H^1(\Omega_\delta)} \le C \left( F+ F^\theta M^{1-\theta}\right) 
\end{equation}
where
\begin{equation*}
F := \|f\|_{L^2(\Omega_0)} +\sum_{k=0}^{3}\left(\|g^k\|_{H^{\frac{7}{2}-k}(\Gamma)} \right), \qquad M:= \|u\|_{H^1(\Omega_0)}.
\end{equation*}
\end{theorem}
   
Apart from the Carleman estimate, the proof of this above result relies on the use of some Caccioppoli-type interior estimate as well. For instance, denoting by $B_{r}$ a ball of radius $r$, centered at $0$, we show that, if $\sum_{j=1}^nD_{x_j}^4u =0$ in $B_1$, then for fixed $r,\rho\in (0,1)$ with $r>\rho$:
\begin{equation}
\label{cC}
\int\displaylimits_{B_r\setminus\overline{B}_{\varrho}}{(|D^2u|^2+|D^3u|^2)}\lesssim \frac{1}{(r-\varrho)^2}\int\displaylimits_{B_{2r}\setminus\overline{B}_{\frac{\varrho}{2}}}{(|u|^2+|Du|^2)}.
\end{equation}
Note that, even to bound the second order term $D^2u$ only, we need $H^1$-norm of $u$ on the right hand side, i.e.
\begin{equation*}
\int\displaylimits_{B_r\setminus\overline{B}_{\varrho}}|D^2u|^2\lesssim \frac{1}{(r-\varrho)^2}\int\displaylimits_{B_{2r}\setminus\overline{B}_{\frac{\varrho}{2}}}{(|u|^2+|Du|^2)}.
\end{equation*}
However, in the case of $\widetilde{u}$ solving $(-\Delta)^2\widetilde{u}=0$ in $B_1$, it is possible to bound $D^2\widetilde{u}$ by the $L^2$-norm of $\widetilde{u}$ only (see \cite{BAMS}):
\begin{equation*}
\int\displaylimits_{B_r\setminus\overline{B}_{\varrho}}|D^2\widetilde{u}|^2\lesssim \frac{1}{(r-\varrho)^2}\int\displaylimits_{B_{2r}\setminus\overline{B}_{\frac{\varrho}{2}}}|\widetilde{u}|^2.
\end{equation*}
Thus, the Caccioppoli estimate \eqref{cC} suggests to consider the $H^1$-norm as the natural candidate instead of the $L^2$-norm for the above theorem.

Next we talk about the strong unique continuation principle (SUCP). If a solution $u$ of the equation $\mathcal{L}_{A,q}u =0$ in $\Omega$ vanishes to infinite order at $x_0\in\Omega$ in the sense that
\begin{equation*}
\lim_{r\to 0}\frac{1}{r^m}\left( \int\displaylimits_{B(x_0,r)}{(|u|^2+ |\nabla u|^2)}\right)^{1/2} =0 \quad \text{ for all } \ m\ge 0,
\end{equation*}
then we say the SUCP holds for this operator if $u=0$ in $\Omega$ is the only solution. 

Concerning the SUCP, we have a very interesting observation to announce. We find that this property is dimension dependent. In three and higher dimensions, it does not hold. However in two dimension, due to elliptic factorization of our operator it holds. 
We begin with recalling a result by \cite{alinhac} which asserts that:\\
 
\noindent
In $\R^n$, $n\ge 2$, let $P=P(x,y,t,D_x,D_y,D_t)$, $t\in \R^{n-2}$ be a differential operator of order $m$, $m\ge 2$ with principal symbol $p_m(x,y,t,\xi,\eta,\tau)$ and $M$ be a sub-manifold of co-dimension $2$. If the principle part $p_m(0,0,0,1,\eta, 0)$ has two roots which are non-real and non-conjugate, then there exists a neighbourhood $V$ of $0$ and two functions $a, u\in C^\infty(V)$ which vanishes of all order on $M\cap V$ and satisfies $Pu-au=0$ in $V$.\\

\noindent
Since in $\R^3$ (or $n\geq 3$), the sub-manifold of co-dimension $2$ is given by lines, the above property precisely corresponds to the vanishing of infinite order at $0$. Our operator $\sum_{i=1}^3 D_{x_i}^4$ satisfies all the hypothesis of the above theorem, since $1+\eta^4 =0$ has two roots which are non-real and non-conjugate, which concludes that the operator $\sum D_{x_i}^4 -aI$ does not have the strong unique continuation property. It is a strike difference with the general second order elliptic operators and the bi-harmonic operator for which SUCP is always true. 

On the other hand, if we consider the $2$-dimension case, the above result no longer applies. Now as the principal part of our operator can be written as a product of elliptic operators of second order
\begin{equation*}
D_1^4 + D_2^4 = (D_1^2+D_2^2 - \sqrt{2} D_1 D_2) (D_1^2+D_2^2 + \sqrt{2} D_1 D_2),
\end{equation*}
the result of \cite{koch} ensures the strong unique continuation principle in this situation.

Finally, we briefly describe the plan of the rest of the paper. In Section \ref{S2}, we derive the Carleman estimates and as an immediate application we show the UCP across hyperplane and hypersurface. In Section \ref{S3}, we prove the weak UCP (Theorem \ref{T1}) and the UCP for local Cauchy data (Theorem \ref{T2}). In the final Section \ref{S4}, as an application of the Carleman estimates derived in Section \ref{S2}, we prove the stability estimate (Theorem \ref{T4}). 

\section{Carleman estimate}
\label{S2}
\setcounter{equation}{0}
This section is dedicated to build Carleman estimates. 
Let us introduce some standard notations which is used through out the paper. 
 Let $u,\vartheta\in L^2(\Omega)$. We write
\[(u\, |\, \vartheta)\ =\ \int\displaylimits_{\Omega}u\,\overline{\vartheta}\ dx, \quad \|u\|_{L^2}\ =\ (u\, |\, u)^{1/2}.\]
We say that the estimate
\[F_1(u;h)\ \lesssim\ F_2(u;h)\]
holds for all $u$ belonging to some function space and for
$0<h<1$ small enough, if there exists constant $C > 0$, independent of $h$ but depends on $\Omega, A, q$ and $n$, such that 
the inequality $F_1 (u; h) \leq C F_2(u;h)$ is satisfied. We follow the convention that $C$ is an unspecified positive constant which may vary among inequalities, but not across equalities. Generally $C$ depends on various parameters which is specified when necessary. 
We first announce the following Carleman estimate with the linear weight.
\begin{lemma}[(Carleman inequality with linear weight)]
\label{L1}
Let $\Omega = \{x=(x',x_n)\in \R^n: a<x_n <b\}$ for some $a, b\in\mathbb{R}$.
Then the Carleman estimate
\begin{equation}
\label{carleman}
h\|w\|_{L^2(\Omega)} \lesssim \ \|e^{\frac{x_n}{h}}h^4\mathcal{L}_{A,q}(e^{-\frac{x_n}{h}}w)\|_{L^2(\Omega)} 
\end{equation}
holds for all $w\in C_c^{\infty}(\Omega)$ and $0<h<1$ small enough.  
\end{lemma}

Let us assume for the moment that the above lemma holds true. We would like to motivate the readers how one uses such estimates to derive certain UCP results. We derive the following simple UCP across a hyperplane with the help of the above estimate. 

\begin{proposition}[(UCP across a hyperplane)]
\label{P1}
Let $\Omega = \{x=(x',x_n)\in \R^n: a<x_n <b\}$ for some $a, b\in\mathbb{R}$ and assume that $u\in H^4(\Omega)$ satisfies
\[
\mathcal{L}_{A,q} u =0 \quad \text{ in } \ \Omega.
\]
If $u\arrowvert_{b-\varepsilon <x_n<b} =0$ for some $\varepsilon>0$, then $u \equiv 0$ in $\Omega$.
\end{proposition}

\begin{proof}
We have that $\Omega = \{x=(x',x_n)\in \R^n: a<x_n <b\}$ and $u\in H^4(\Omega)$ satisfies
\begin{equation*}
\begin{cases}
\mathcal{L}_{A,q} u =0 \quad &\text{ in } \ \Omega\\
u=0 \quad &\text{ in } \ {b-\varepsilon <x_n<b}.
\end{cases}
\end{equation*}
It is enough to show that $u=0$ in $c_0<x_n<b$ where $c_0$ is any number satisfying $a<c_0<b$.

We rewrite the estimate (\ref{carleman}) as,
\begin{equation*}
\|e^{\frac{x_n}{h}}w\|_{L^2(\Omega)} \lesssim \ h^3 \|e^{\frac{x_n}{h}}\mathcal{L}_{A,q}w\|_{L^2(\Omega)} 
\end{equation*}
which holds for all $w\in H^4_0(\Omega)$ and for $0<h<1$ sufficiently small.
Now we choose $w=\chi u$ where $\chi(x^\prime ,x_n) = \zeta(x_n)$ for some $\zeta\in C^\infty(\R)$ satisfying $\zeta =1$ for $t\ge c_0$ and $\zeta =0$ near $t\le a$. Since $u=0$ near $x_3 =b$ and $\chi =0$ near $x_3 =a$, we have that $w\in H^4_0(\Omega)$. 
Therefore,
\begin{align}
\|e^{\frac{x_n}{h}}u\|_{L^2(\{c_0<x_n<b\})} &\le \|e^{\frac{x_n}{h}} (\chi u)\|_{L^2(\Omega)}\nonumber\\
& \lesssim \ h^3 \|e^{\frac{x_n}{h}}\mathcal{L}_{A,q} (\chi u)\|_{L^2(\Omega)} \nonumber\\
& \lesssim h^3 \left( \|e^{\frac{x_n}{h}}\chi \mathcal{L}_{A,q}\,u\|_{L^2(\Omega)} + \|e^{\frac{x_n}{h}}[\mathcal{L}_{A,q},\chi]u\|_{L^2(\Omega)}\right)
\label{1} 
\end{align}
where $[\mathcal{L}_{A,q},\chi]u := u\ D^4_j \chi +4 D_j u \ D_j^3 \chi + 6D_j^2 u \ D^2_j \chi + 4 D^3_j u \ D_j \chi$ is the commutator term. We observe that, $\textrm{supp}\,[\mathcal{L}_{A,q},\chi]u \subseteq \textrm{supp}\, (\nabla \chi) \subseteq  \{a\le x_n \le c_0\}$. Then using $\mathcal{L}_{A,q} u =0$ in $\Omega$, the inequality (\ref{1}) implies
\begin{equation*}
\|e^{\frac{x_n}{h}}u\|_{L^2(\{c_0<x_n<b\})} \lesssim h^3 \|e^{\frac{x_n}{h}}[\mathcal{L}_{A,q},\chi]u\|_{L^2(\{a\le x_n \le c_0\})}.
\end{equation*}
But $e^{\frac{x_n}{h}}\le e^{\frac{c_0}{h}}$ when $x_n \le c_0$ and $e^{\frac{x_n}{h}} \ge e^{\frac{c_0}{h}}$ when $x_n\ge c_0$. This yields
\begin{equation*}
\begin{aligned}
e^{\frac{c_0}{h}}\|u\|_{L^2(\{c_0<x_n<b\})} \le \|e^{\frac{x_n}{h}}u\|_{L^2(\{c_0<x_n<b\})}
& \lesssim h^3 \|e^{\frac{x_n}{h}}[\mathcal{L}_{A,q},\chi]u\|_{L^2(\{a\le x_n \le c_0\})}\\
& \lesssim h^3 e^{\frac{c_0}{h}} \|[\mathcal{L}_{A,q},\chi]u\|_{L^2(\{a\le x_n \le c_0\})}.
\end{aligned}
\end{equation*}
Since $[\mathcal{L}_{A,q},\chi]u$ is a fixed function, dividing by $e^{\frac{c_0}{h}}$ and letting $h\to 0$ shows that
\[
\|u\|_{L^2(\{c_0<x_n<b\})} =0
\]
which completes the proof.
\hfill
\end{proof}

Now we prove the Lemma \ref{L1}.

Let $\widetilde{\Omega}\subset\mathbb{R}^n$ be a non-empty open set and $\varphi\in C^{\infty}(\widetilde{\Omega};\mathbb{R})$ with $\nabla\varphi\neq 0$ be some phase function. 
Let us first consider the principal part of the semi classical operator $h^4\mathcal{L}_{A,q}(x,D)$, say $P(x,hD)$ as
\[
P(x,hD) =\  h^4\sum_{j=1}^n D^4_{x_j}; \quad h^4\mathcal{L}_{A,q}(x,D) =\ P + h^3A\cdot hD + h^4q.
\]
The operator $P$ conjugated with the exponential weight $e^{\frac{\varphi}{h}}$ is denoted as 
\[
P_{\varphi} :=\ e^{\frac{\varphi}{h}}(\sum_{j=1}^n h^4D^4_{x_j})e^{-\frac{\varphi}{h}} =  \sum_{j=1}^n (hD_{x_j} + i\partial_{x_j}\varphi)^4 =\ \mathcal{A}+i\mathcal{B} \quad\mbox{(say)}
\]
with its semi classical symbol $p_{\varphi}(x,\xi)$ given by
\[
p_{\varphi}(x,\xi) =\ \sum_{j=1}^n (\xi_j + i\partial_{x_j}\varphi)^4=\ a(x,\xi) + ib(x,\xi), \quad (x,\xi)\in (\widetilde{\Omega}\times\mathbb{R}^n)
\]
where $a(x,\xi)$ and $b(x,\xi)$ denote the Weyl symbols of the semi-classical operators $A$ and $B$ respectively with the usual summation convention:
\begin{equation}
\label{f-symbol}
a(x,\xi)= \xi_j^4 - 6(\partial_{x_j}\varphi)^2 \xi_j^2 + (\partial_{x_j}\varphi)^4 \ \ \mbox{ and } \ \ b(x,\xi) = 4(\partial_{x_k}\varphi)\xi_k^3 - 4(\partial_{x_k}\varphi)^3\xi_k.
\end{equation}
The Poisson bracket of these two symbols is given by
\begin{align}
&\ \{a,b\}(x,\xi) \nonumber\\
:=&\ a^{\prime}_{\xi}\cdot b^{\prime}_x - a^{\prime}_x\cdot b^{\prime}_{\xi} \nonumber \\
=&\ \{4\xi_j^3 - 12\xi_j(\partial_{x_j}\varphi)^2\}\cdot \{4\xi_k^3(\partial^2_{x_jx_k}\varphi)- 12\xi_k(\partial_{x_k}\varphi)^2(\partial^2_{x_jx_k}\varphi)\} \nonumber\\ 
& -\{-12\xi_j^2(\partial_{x_j}\varphi)(\partial^2_{x_jx_k}\varphi) + 4(\partial_{x_j}\varphi)^3(\partial^2_{x_jx_k}\varphi)\}\cdot \{12\xi_k^2(\partial_{x_k}\varphi) -4 (\partial_{x_k}\varphi)^3)\} \nonumber\\
=&\ \{ 16\xi_j^3\xi_k^3 - 48\xi_j^3\xi_k(\partial_{x_k}\varphi)^2- 48\xi_j\xi_k^3(\partial_{x_j}\varphi)^2 + 144\xi_j\xi_k(\partial_{x_j}\varphi)^2(\partial_{x_k}\varphi)^2 \nonumber\\
& +144\xi_j^2\xi_k^2(\partial_{x_j}\varphi)(\partial_{x_k}\varphi)-48\xi_j^2(\partial_{x_j}\varphi)(\partial_{x_k}\varphi)^3-48(\partial_{x_j}\varphi)^3\xi_k^2(\partial_{x_k}\varphi) \nonumber\\
& + 16(\partial_{x_j}\varphi)^3(\partial_{x_k}\varphi)^3 \}\ (\partial^2_{x_jx_k}\varphi) \label{pb}.
\end{align}
We want this Poisson bracket to be 
\begin{equation}\label{new5} \{a,b\}(x,\xi)\geq 0 \end{equation}
on the set
\begin{align}
a(x,\xi) &= (\xi_j^4  + (\partial_{x_j}\varphi)^4) -\ 6 (\partial_{x_j}\varphi)^2 \xi_j^2 =0 \label{ab1}\\
\mbox{ and }\quad b(x,\xi) &= 4(\partial_{x_j}\varphi)\xi_j^3 - 4(\partial_{x_j}\varphi)^3\xi_j =0. \label{ab2} 
\end{align}
If $\{a,b\}(x,\xi)>0$ over the set $a(x,\xi)=b(x,\xi)=0$, then such weights are known to be satisfying the sub-ellipticity condition connecting the symbol $p(x,\xi)$ of the operator $P(x,D)$ and a weight function $\varphi$, see \cite{H63,Hor4}. And if  $\{a,b\}=0$ over $a=b=0$, then such weights are known as limiting Carleman weights.  

For example, if we choose  $\varphi(x)= (\varrho\cdot x)$ for some $\varrho \in \mathbb{R}^n\setminus \{0\}$ a non-zero constant vector,
then the Poisson bracket becomes zero. However, if we choose $\varphi= x_n^2$ then it satisfies the sub-ellipticity condition.

Now we introduce the idea of convexification of the weight functions. Let us choose some $\varphi$ such that \eqref{new5} holds, i.e.  $\{a,b\}(x,\xi)\geq 0$  on the set $a(x,\xi) = b(x,\xi) = 0$. Note that it does not satisfy the sub-ellipticity condition mentioned above.    
Let us replace $\varphi$ by $f(\varphi)$, where
$f^\prime > 0$ and $\frac{f^{\prime\prime}}{f^\prime}>0 $ sufficiently large.
We denote $\psi(x) = f(\varphi(x))$ which is known as the convexified weight function of $\varphi$. 
We denote by $\widetilde{a}$ and $\widetilde{b}$ be the new corresponding symbols. Let us note that 
\[\partial_{x_j}\psi = f^\prime(\varphi)\partial_{x_j}\varphi,\qquad
\partial^2_{x_jx_k}\psi =f^{\prime\prime}(\varphi)\partial_{x_j}\varphi\partial_{x_k}\varphi + f^\prime(\varphi)\partial^2_{x_jx_k}\varphi.
\]
If $\xi$ satisfies \eqref{ab1} and \eqref{ab2}, then it is natural to replace $\xi$ by 
$\eta = f^{\prime}(\varphi)\xi$ in order to preserve the conditions \eqref{ab1} and \eqref{ab2} 
for the new symbol. So, here we make two substitutions $\varphi \mapsto \psi = f(\varphi(x))$ and $\xi \mapsto \eta = f'(\varphi(x))\xi$ in \eqref{pb} 
which becomes, when restricted to $\widetilde{a}(x,\eta) =\widetilde{b}(x,\eta) = 0$,
\begin{align}
&\ \{\widetilde{a},\widetilde{b}\}(x,\eta) \nonumber\\[1mm]
=&\ \left\lbrace  16\xi_j^3\xi_k^3 - 48\xi_j^3\xi_k(\partial_{x_k}\varphi)^2- 48\xi_j\xi_k^3(\partial_{x_j}\varphi)^2 + 144\xi_j\xi_k(\partial_{x_j}\varphi)^2(\partial_{x_k}\varphi)^2 \right. \nonumber\\
& +144\xi_j^2\xi_k^2(\partial_{x_j}\varphi)(\partial_{x_k}\varphi)-48\xi_j^2(\partial_{x_j}\varphi)(\partial_{x_k}\varphi)^3-48(\partial_{x_j}\varphi)^3\xi_k^2(\partial_{x_k}\varphi) \nonumber\\
& \left. + 16(\partial_{x_j}\varphi)^3(\partial_{x_k}\varphi)^3 \right\rbrace \, (f^\prime(\varphi))^6\,\left(f^{\prime\prime}(\varphi)\partial_{x_j}\varphi\partial_{x_k}\varphi + f^\prime(\varphi)\partial^2_{x_jx_k}\varphi\right)\nonumber\\[1mm] 
=& 16\left(\xi_j^3(\partial_{x_j}\varphi)\right)^2(f^\prime(\varphi))^6\,f^{\prime\prime}(\varphi) 
- 96\left(\xi_j^3(\partial_{x_j}\varphi)\right)^2(f^\prime(\varphi))^6\,f^{\prime\prime}(\varphi)   + 144\left(\xi_j^3(\partial_{x_j}\varphi)\right)^2(f^\prime(\varphi))^6\,f^{\prime\prime}(\varphi) \nonumber\\
& +144\left(\xi_j^2(\partial_{x_j}\varphi)^2\right)^2 (f^\prime(\varphi))^6\,f^{\prime\prime}(\varphi) 
-96\left(\xi_j^2(\partial_{x_j}\varphi)^2\right)\left((\partial_{x_j}\varphi)^4\right)(f^\prime(\varphi))^6\,f^{\prime\prime}(\varphi)\nonumber\\
&+ 16\left((\partial_{x_j}\varphi)^4\right)^2(f^\prime(\varphi))^6\,f^{\prime\prime}(\varphi) + (f^\prime(\varphi))^7\, \{a,b\}(x,\xi)\nonumber\\[1mm]
=&\ 64\left(\xi_j(\partial_{x_j}\varphi)^3\right)^2(f^\prime(\varphi))^6\,f^{\prime\prime}(\varphi) + 4\left( \xi_j^4 - (\partial_{x_j}\varphi)^4\right)^2(f^\prime(\varphi))^6\,f^{\prime\prime}(\varphi) +(f^\prime(\varphi))^7\, \{a,b\}(x,\xi).\label{new1}
\end{align}
We use relations \eqref{ab1}, \eqref{ab2} to deduce the last line. Now by using \eqref{ab2} again, we write
\begin{align}
64\left(\xi_j(\partial_{x_j}\varphi)^3\right)^2=\ 16\left(\xi_j^3(\partial_{x_j}\varphi) + \xi_j(\partial_{x_j}\varphi)^3\right)^2 =&\ 16\left(\xi_j(\partial_{x_j}\varphi)(\xi_j^2+(\partial_{x_j}\varphi)^2)\right)^2\nonumber\\[1mm]
\geq&\ 64 \left( \xi_j^2(\partial_{x_j}\varphi)^2 \right)^2\nonumber\\[1mm]
=&\ \frac{16}{9} \left((\xi_j^4 + (\partial_{x_j}\varphi)^4)\right)^2 \ \ \mbox{ (by }\eqref{ab1} )\nonumber\\[1mm]
\geq&\ \frac{16}{9}\left((\partial_{x_j}\varphi)^4\right)^2 \ > 0.\label{new2}
\end{align}
Therefore from \eqref{new1}, \eqref{new2} we see that when $\varphi$, satisfying \eqref{new5}, is replaced by the convexified weight function $\psi=f(\varphi)$, where $f^\prime, f^{\prime\prime}>0$, we obtain  
\begin{equation}
\label{new3}
\{\widetilde{a},\widetilde{b}\}(x,\eta)  > \frac{16}{9}(f^\prime(\varphi))^6\,f^{\prime\prime}(\varphi)\, \left((\partial_{x_j}\varphi)^4\right)^2
\end{equation}
which is strictly positive. 

The idea of covexfication will be crucially used in order to derive the Carleman estimates for those weight functions satisfying \eqref{new5}. 
At this end, we introduce the semi classical Sobolev space of order one $H^1_{scl}(\Omega)$ associated with its norm
\[
\|u\|^2_{H^1_{scl}(\Omega)} = \|u\|^2_{L^2(\Omega)} + \|h\nabla u\|^2_{L^2(\Omega)}.
\]
In general one defines the semi-classical Sobolev spaces $H^s(\mathbb{R}^n)$, with $s\in\mathbb{R}$ equipped with the norm 
\[
\|u\|_{H^s(\mathbb{R}^n)} = \|{\langle hD\rangle}^su\|_{L^2}\ \mbox{ where }\ \langle \xi \rangle = (1+|\xi|^2)^{\frac{1}{2}}.
\]
We begin with the following $H^1_{scl}$ Carleman estimate which does not involve the boundary terms.

\begin{proposition}
\label{P2}
Let $\Omega\Subset\widetilde{\Omega}$ are two open subsets of $\mathbb{R}^n$. Let $\varphi\in C^{\infty}(\widetilde{\Omega};\mathbb{R})$ such that \eqref{new5} is satisfied.
Then the Carleman estimate
\begin{equation}
\label{c-e}
 h^2\|w\|^2_{H^1_{scl}(\Omega)} \lesssim \ \|e^\frac{\varphi}{h}h^4\mathcal{L}_{A,q}(e^{-\frac{\varphi}{h}}w)\|^2_{L^2(\Omega)} 
\end{equation}
holds for all $w\in C_c^{\infty}(\Omega)$ and $0<h<1$ small enough. 
\end{proposition}
\begin{proof}
The proof is divided into two parts: using the notation as before, we will show first
\begin{equation}
\label{c-e-1}
h^2\|w\|^2_{H^1_{scl}(\Omega)} \lesssim  \|P_{\varphi}w\|^2_{L^2(\Omega)}, \quad w\in C^{\infty}_c(\Omega)
\end{equation}
and then we add the lower order terms into it to get the desired estimate \eqref{c-e}.

Let us write
\begin{equation*}
\begin{aligned}
P_{\varphi} &= \lb h^4 D_{x_j}^4 - 6h^2(\partial_{x_j}\varphi)^2D_{x_j}^2 + (\partial_{x_j}\varphi)^4 \rb + i \lb 4 h^3(\partial_{x_j}\varphi)D_{x_j}^3 - 4 h(\partial_{x_j}\varphi)^3D_{x_j}\rb \\[1mm]
&= \mathcal{A}+i\mathcal{B}, \ \mbox{say}.
\end{aligned}
\end{equation*}
Then for $w\in C_c^{\infty}(\Omega)$,
\begin{align*}
\|P_{\varphi}w\|^2_{L^2} &=((\mathcal{A}+i\mathcal{B})w\, |\, (\mathcal{A}+i\mathcal{B})w)\\[1mm]
& = \|\mathcal{A}w\|^2_{L^2} + \|\mathcal{B}w\|^2_{L^2} + i(\mathcal{B}w\, |\, \mathcal{A}w) - i(\mathcal{A}w | \mathcal{B}w).
\end{align*}
The standard Weyl quantizations gives the commutator term as
\[ i[\mathcal{A},\mathcal{B}] := i\left(\mathcal{A}\mathcal{B} - \mathcal{B}\mathcal{A}\right)= \mbox{Op}_h(h\{a,b\}).\]

For the moment, let us consider a particular case when $\varphi(x)=(\varrho\cdot x)$ for some $\varrho\in\mathbb{R}^n$ non-zero vector. We know that in this case the Poisson bracket becomes zero. 
Also, in this case,  
$\mathcal{A}_{(\varrho\cdot x)}= \lb h^4 D_{x_j}^4 - 6h^2\varrho_j^2D_{x_j}^2 + \varrho_j^4 \rb$ 
and $\mathcal{B}_{(\varrho\cdot x)}=\lb 4 h^3\varrho_jD_{x_j}^3 - 4 h\varrho_j^3D_{x_j}\rb $ are constant coefficient self-adjoint operators.
Thus the commutator
term $i[\mathcal{A}_{(\varrho\cdot x)},\mathcal{B}_{(\varrho\cdot x)}]$ acting on $C_c^{\infty}(\Omega)$ always satisfy
\[i[\mathcal{A}_{(\varrho\cdot x)},\mathcal{B}_{(\varrho\cdot x)}] 
= 0.\]
Therefore,
\[
\|P_{(\varrho\cdot x)}w\|^2_{L^2} =\|\mathcal{A}_{(\varrho\cdot x)}w\|^2_{L^2} + \|\mathcal{B}_{(\varrho\cdot x)}w\|^2_{L^2}.
\]
Now, for any $w\in C^{\infty}_c(\Omega),$
\begin{align*}
(\mathcal{A}_{(\varrho\cdot x)}w\, |\, w)=& \left( \lb h^4 D_{x_j}^4 - 6h^2\varrho_j^2D_{x_j}^2 + \varrho_j^4 \rb w\ |\ w\right) \\ 
=& [ \ h^4(D_{x_j}^2w\ |\ D_{x_j}^2w) -6h^2\varrho_j^2(D_{x_j}w | D_{x_j}w) + \varrho_j^4(w\ |\ w) ].
\end{align*}
By using the inequality $|\alpha\beta| \leq \frac{\delta}{2}|\alpha|^2 + \frac{1}{2\delta}|\beta|^2$ on the left hand side and
using the Poincar\'{e} inequality on the first term of the right hand side, we then obtain,
\[
\frac{1}{2\varrho_j^4}\|\mathcal{A}w\|^2_{L^2} + \frac{\varrho_j^4}{2}\|w\|^2_{L^2}\ \geq h^2\|hDw\|^2_{L^2} - \mathcal{O}(h^2)\|Dw\|^2_{L^2} + \varrho_j^4\|w\|^2_{L^2}.
\]
Consequently, we get 
\begin{equation}
\label{Au}
h^2\|w\|^2_{H^1_{scl}} \lesssim \|\mathcal{A}_{(\varrho\cdot x)}w\|^2_{L^2} + h^2\|Dw\|^2_{L^2}.
\end{equation}
Now we could try to use that $\mathcal{B}$ is associated to two non-vanishing gradient fields to obtain
\[
\|\mathcal{B}_{(\varrho\cdot x)}w\|_{L^2} \geq \mathcal{O}(h)\|Dw\|_{L^2} - \mathcal{O}(h^3)\|D^3w\|_{L^2}.
\]
But it is not good enough to absorb the $\mathcal{O}(h^2)\|Dw||^2$ term in \eqref{Au} to obtain \eqref{c-e-1}. We seek for the idea of convexification of the weight function to establish such estimates.

In general, for any $\varphi$ satisfying $\{a,b\}\geq 0$ whenever $a=b=0$,  we convexify the weight function $\varphi$ and introduce $\psi=f(\varphi)$, where $f(\lambda) =\lambda + \frac{h}{2\epsilon}\lambda^2$, $\lambda\in\mathbb{R}$, i.e.
\begin{equation}
\label{conwei}
\psi = \varphi + \frac{h}{2\epsilon}\varphi^2\quad\mbox{in }\widetilde{\Omega}
\end{equation} 
with $\epsilon$ a suitable small parameter to be chosen independent of $h$ and $0<h<\epsilon<1$. 

We denote by $\widetilde{a}$ and $\widetilde{b}$ be the new corresponding symbols and by $\widetilde{\mathcal{A}}$
and $\widetilde{\mathcal{B}}$ be the corresponding operators when $\varphi$ is replaced by $\psi$. 

Let $\eta=(1+\frac{h}{\epsilon}\varphi)\xi$ and we deduce (cf. \eqref{new1} and \eqref{new3}), whenever $\widetilde{a}(x,\eta)=\widetilde{b}(x,\eta)=0$,
\begin{equation}
\begin{aligned}
\label{new4}
\{\widetilde{a},\widetilde{b}\}(x,\eta)=&\, 64\,\left(\xi_j(\partial_{x_j}\varphi)^3\right)^2\frac{h}{\epsilon}(1+\frac{h}{\epsilon}\varphi)^6 + 4\,\left( \xi_j^4 - (\partial_{x_j}\varphi)^4\right)^2\frac{h}{\epsilon}(1+\frac{h}{\epsilon}\varphi)^6\\
& \hspace{4.3cm}+(1+\frac{h}{\epsilon}\varphi)^7\, \{a,b\}(x,\xi) \ = d(x,\xi)\ \mbox{ (say)}
\end{aligned}
\end{equation}
with 
\[d(x,\xi) \geq \frac{16}{9}\,\frac{h}{\epsilon}(1+\frac{h}{\epsilon}\varphi)^6\left((\partial_{x_j}\varphi)^4\right)^2 \ > 0 .\]

Now as we see that on the $x$-dependent surface in $\eta$-space, given by $\widetilde{b}(x, \eta) = 0$,
the fourth order polynomial $\{\widetilde{a}, \widetilde{b}\}(x,\eta)$ becomes positive when
$\widetilde{a}(x,\eta)=\eta_j^4 - 6\eta_j^2(\partial_j\psi)^2 +(\partial_j\psi)^4=0$. Thus for some $c\in C^{\infty}(\Omega;\mathbb{R})$,
\[\{\widetilde{a},\widetilde{b}\}(x, \eta)=\ d(x,\xi) +  c(x)\widetilde{a}(x,\eta),\quad\mbox{whenever }\widetilde{b}(x,\eta) = 0.\]
Then we consider
\[\{\widetilde{a},\widetilde{b}\}(x, \eta) -d(x,\xi) - c(x)\widetilde{a}(x,\eta)\]
which is a fourth order polynomial in $\eta$, vanishing when $\widetilde{b}(x,\eta)=\sum_j 4(\partial_{x_j}\psi)\eta_j^3 - \sum_j 4(\partial_{x_j}\psi)^3\eta_j =0$.
Thus it is of the form $l(x,\eta)\widetilde{b}(x,\eta)$ where $l(x,\eta)$ is affine in $\eta$ with smooth
coefficients and hence we end up with
\begin{equation}
\label{commu}
\{\widetilde{a},\widetilde{b}\}(x,\eta)=\ d(x,\xi) + c(x)\widetilde{a}(x,\eta) + l(x,\eta)\widetilde{b}(x,\eta).
\end{equation}
On the other hand, we have the standard Weyl quantizations
\begin{align*}
 \mbox{Op}_h(c\widetilde{a})\ =&\ \frac{1}{2}c\circ \widetilde{\mathcal{A}} + \frac{1}{2}\widetilde{\mathcal{A}}\circ c + h^4q_1(x) \\
 \mbox{Op}_h(l\widetilde{b})\ =&\ \frac{1}{2}L\widetilde{\mathcal{B}} + \frac{1}{2}\widetilde{\mathcal{B}}L + h^4q_2(x),
\end{align*}
where $q_j$'s ($j=1,2$) are smooth functions which together with their derivatives are bounded
uniformly with respect to $\epsilon$ near $\overline{\Omega}$.
Now the commutator term is given by
\[ i[\widetilde{\mathcal{A}},\widetilde{\mathcal{B}}]=\ \mbox{Op}_h(h\{\widetilde{a},\widetilde{b}\}).\]
From \eqref{new4} we would like to write,
\begin{align}
\{\widetilde{a},\widetilde{b}\}(x,\eta) \geq \ &\ 64\frac{h}{\epsilon}(1+\frac{h}{\epsilon}\varphi)^6\ \left(\xi_j(\partial_{x_j}\varphi)^3\right)^2\notag\\
= \ & 32\,\frac{h}{\epsilon}(1+\frac{h}{\epsilon}\varphi)^6\ \left(\xi_j(\partial_{x_j}\varphi)^3\right)^2 + \ 32\,\frac{h}{\epsilon}(1+\frac{h}{\epsilon}\varphi)^6\ 8\left(\xi_j(\partial_{x_j}\varphi)^3\right)^2 \notag\\
> \ & 32\,\frac{h}{\epsilon}(1+\frac{h}{\epsilon}\varphi)^4\left(\eta_j(\partial_{x_j}\varphi)^3\right)^2+ \frac{8}{9}\,\frac{h}{\epsilon}(1+\frac{h}{\epsilon}\varphi)^6\left((\partial_{x_j}\varphi)^4\right)^2\notag\\
= \ & \frac{h}{\epsilon}\, \widetilde{d}(x,\eta)\,\, \mbox{(say)} . \label{newd}
\end{align}
Thus we have from \eqref{commu} and \eqref{newd},
\[
h\{ \widetilde{a},\widetilde{b}\}(x,\eta) \geq \frac{h^2}{\epsilon}\widetilde{d}(x,\eta)+ hc(x)\widetilde{a}(x,\eta) + hl(x,\eta)\widetilde{b}(x,\eta),\ \ (x,\eta)\in (\widetilde{\Omega}\times \mathbb{R}^n).
\]
Now suppose that $0< h\ll \epsilon \ll 1$. Since $\widetilde{d}$ is elliptic and of order $2$, there is a constant
$\widetilde{c}_{\widetilde{\Omega}}> 0$ independent of $\epsilon$, such that
\[
\widetilde{d}(x,\eta)\ \geq\ \widetilde{c} \langle \eta \rangle^2, \quad x \mbox{ near }\overline{\Omega},\ \ \eta\in\mathbb{R}^n.
\]
Then by using the G\aa{}rding inequality one simply gets
\[
(\widetilde{D}w\, |\, w) \ \geq \frac{\widetilde{c}}{2}\|w\|^2_{H^1_{scl}}, \quad w\in C_c^{\infty}(\Omega)\ \mbox{ and }\ h\ \mbox{ is small enough}.
\]
Thus on the operator level it implies that
\begin{align*}
i([\widetilde{\mathcal{A}},\widetilde{\mathcal{B}}]w\, |\, w ) \geq&  \frac{h^2}{\epsilon}(\widetilde{D} w\, |\, w) + h\Re(c\widetilde{\mathcal{A}}w\, |\,w) + h\Re(\widetilde{\mathcal{B}}w\, |\, Lw) +h^5((q_1+q_2)w\, |\,w)\\
\geq&\ \frac{\widetilde{c}}{2}\frac{h^2}{\epsilon}\|w\|^2_{H^1_{scl}}- \underbrace{Ch(\|\widetilde{\mathcal{A}}w\|_{L^2}\|w\|_{L^2} + \|\widetilde{\mathcal{B}}w\|_{L^2}\|hDw\|_{L^2})}_{\leq\ \frac{1}{2}\|\widetilde{\mathcal{A}}w\|^2_{L^2}+\frac{1}{2}\|\widetilde{\mathcal{B}}w\|^2_{L^2}+\frac{C_1h^2}{2}(\|w\|^2_{L^2}+\|hD w\|^2_{L^2})} -\mathcal{O}(h^5)\|w\|^2_{L^2}.
\end{align*}
Now when $0 < h \ll \epsilon\ll 1$, we obtain
\begin{align*}
\|P_{\psi}w\|^2_{L^2}= \|(\widetilde{\mathcal{A}}+i\widetilde{\mathcal{B}})w\|^2_{L^2}\ = \|\widetilde{\mathcal{A}}w\|^2_{L^2} + \|\widetilde{\mathcal{B}}w\|^2_{L^2} + i([\widetilde{\mathcal{A}},\widetilde{\mathcal{B}}]w\, |\,w) \geq C_2\frac{h^2}{\epsilon}\|w\|^2_{H^1_{scl}}.
\end{align*}
Furthermore, since $e^{\frac{\varphi^2}{2\epsilon}}$ and its all derivatives are bounded in $\Omega$ by some constant independent of $h$, with $0 < h \ll \epsilon\ll 1$, we finally get
\begin{equation}
\label{c-e-2}
h^2\|w\|^2_{H^1_{scl}(\Omega)} \lesssim \ \|e^{\frac{\varphi}{h}}h^4\sum_j D^4_{x_j}(e^{-\frac{\varphi}{h}}w)\|^2_{L^2(\Omega)}.
\end{equation}
This completes the first part, namely establishing the result \eqref{c-e-1}. Now we add the lower order terms into \eqref{c-e-2}.

\paragraph{(a) Addition of the zeroth order term $(h^{4}q)$ where $q\in L^{\infty}(\Omega,\mathbb{C})$:}
\[
\|qw\|_{L^2}\ \leq\|q\|_{L^{\infty}}\|w\|_{L^2} \leq \|q\|_{L^{\infty}}\|w\|_{H^{1}_{scl}(\Omega)}.
\]
\paragraph{(b) Addition of the first order term $(h^{4}A\cdot D)$ where $A\in W^{1,\infty}(\Omega,\mathbb{C}^n)$:}
\[
h^{3}e^{\frac{\varphi\cdot x}{h}}(A\cdot hD) e^{-\frac{\varphi}{h}}\ =\ h^{3}( iA\cdot\nabla\varphi +A\cdot hD).
\]
For the first term, we can write
\[ \|(A\cdot\nabla\varphi)w\|_{L^2}\leq\|A\cdot\nabla\varphi\|_{L^{\infty}}\|w\|_{H^{1}_{scl}(\Omega)} = \mathcal{O}(1)\|w\|_{H^1_{scl}(\Omega)}.
\]
Similarly the second term can be estimated as,
\[
\|A\cdot hDw\|_{L^2} \leq \|A\|_{L^{\infty}}\|hDw\|_{L^2} = \mathcal{O}(1)\|w\|_{H^1_{scl}(\Omega)}.
\]
Therefore,
\[
\|e^{\frac{\varphi}{h}}\{h^3(A\cdot hD) + h^4q\} e^{-\frac{\varphi}{h}}w \|_{L^2} \leq \mathcal{O}(h^3)\|w\|_{H^{1}_{scl}(\Omega)}.
\]    
Thus for $0<h\ll 1$ small enough, the above $\mathcal{O}(h^3)$ term gets absorbed into the left hand side of \eqref{c-e-2} to give
\[
h^2\|w\|^2_{H^1_{scl}(\Omega)} \lesssim \|e^{\frac{\varphi}{h}}h^4\mathcal{L}_{A,q}(e^{-\frac{\varphi}{h}}w)\|^2_{L^2(\Omega)}.
\]
This finishes the proof.
\end{proof}

\begin{proof}[Proof of Lemma \ref{L1}]
It directly follows from the above Proposition \ref{P2} by choosing $\varphi(x)=x_n$. 
\end{proof}

Next we prove that if a solution $u$ of $\mathcal{L}_{A,q} u=0$ vanishes on one side of a hypersurface (not necessarily flat) near some point $x_0$, then $u$ vanishes in a neighbourhood of $x_0$.

\begin{proposition}[(UCP across a hypersurface)]
\label{T}
Assume that $x_0\in\Omega$. Let $V$ be a neighbourhood of $x_0$ and $S$ be a $C^\infty$-hypersurface through $x_0$ such that $V = V_+ \cup S \cup V_-$ where $V_+$ and $V_-$ denote the two sides of $S$. If $u\in H^4(V)$ satisfies
\begin{equation*}
\begin{aligned}
\mathcal{L}_{A,q} u = 0 \quad &\text{ in } \ V\\
u = 0 \quad &\text{ in } \ V_+,
\end{aligned}
\end{equation*}
then $u=0$ in some neighbourhood of $x_0$.
\end{proposition}

The Carleman inequality with the linear weight $\pm x_n$ is not sufficient to prove the UCP across a general hypersurface. We need to bend it by considering  quadratic weight functions of the form $\pm x_n + |x^\prime|^2\mp c^2$. Thus we prove the following estimate with convex weight.

\begin{lemma}[(Carleman inequality with quadratic weight)]
\label{L2}
Let $\Omega$ be any bounded open set in $\mathbb{R}^n$. Let $\varphi(x) = \pm x_n+|x^\prime|^2\mp c^2$ be the weight function.  Then the Carleman estimate
\begin{equation}
\label{carleman1}
h\|w\|_{L^2(\Omega)} \lesssim \ \|e^{\frac{\varphi}{h}}h^4\mathcal{L}_{A,q}(e^{-\frac{\varphi}{h}}w)\|_{L^2(\Omega)} 
\end{equation}
holds for all $w\in C_c^{\infty}(\Omega)$ and $0<h<1$ small enough.
\end{lemma}
\noindent
Let us first see how we can derive the Proposition \ref{T} by assuming the Lemma \ref{L2}. 

\begin{proof}
[Proof of Proposition \ref{T}]
We first consider the case $x_0 =0$ and $S = \{x_n =0\}$. Assume that $V = B_{4\delta}$ for some small $\delta >0$ and we have that $u\in H^4(V)$ satisfies
\begin{align}
\mathcal{L}_{A,q} u = 0 \quad &\text{ in } \ V \label{3}\\
u = 0 \quad &\text{ in } \ V\cap \{x_n<0\}\notag.
\end{align}
We will show that $u =0$ in $B_\varepsilon\cap \{x_n>0\}$ for some $\varepsilon >0$.

Let us consider the weight $\varphi_0 (x^\prime, x_n) = -x_n +|x^\prime|^2 + \delta^2$. The level set $\varphi_0^{-1}(0)$ is the parabola $x_n = |x^\prime|^2 + \delta ^2$. Now define the sets
\begin{align*}
W_+ &:= \{ \varphi_0(x)>0 \}\cap \{ x_n>0\}\\
W_- &:= \{-\delta^2< \varphi_0(x)<0 \} \cap \{x_n>0\} .
\end{align*}
It is clear that $W_+$ and $W_-$ are non-empty open sets and 
$B_\varepsilon\cap \{x_n >0\}\subset W_+$ for $\varepsilon = \delta^2$.

\begin{figure}[h]
	\centering
	\begin{tikzpicture}
	\begin{axis}[
	axis lines = center,
	xmin=-5, xmax=5, ymin=-4, ymax=9,
	axis equal,
	xlabel = $x'$,
	ylabel = {$x_n$},
	yticklabels={,,}
	]
	\draw (axis cs: 0, 0) circle [radius=1] node at (1.44,-0.4) {$B_\epsilon$};
	\draw (axis cs: 0, 0) circle [radius=3.5] node at (-4.1,-0.4) {$B_{4\delta}$};
	\draw (0,1.2) parabola (5,7.5) node at (1, 2.1) {$W_{-}$};
	\draw (0,1.2) parabola (-5,7.5) node at (4.3, 1.2) {$W_{+}$};
	\draw (0,2.5) parabola (4,7.5) node at (-1.9, 7) {$\varphi_0=-\delta^2$};
	\draw (0,2.5) parabola (-4,7.5) node at (-5.8, 6) {$\varphi_0=0$};
	\end{axis}
	\end{tikzpicture}
\end{figure}
We rewrite the estimate (\ref{carleman1}) as,
\begin{equation}
\label{carest}
\|e^{\frac{\varphi_0}{h}}w\|_{L^2(\Omega)} \lesssim \ h^3 \|e^{\frac{\varphi_0}{h}}\mathcal{L}_{A,q}w\|_{L^2(\Omega)} 
\end{equation}
which holds for all $w\in H^4_0(\Omega)$ and for $0<h<1$ sufficiently small.
Now we choose $w=\chi u$ where $\chi(x) := \zeta \left( \frac{\varphi_0(x)}{\delta^2}\right)  \eta\left( \frac{|x|}{2\delta}\right) $ where $\zeta, \eta \in C_c^\infty(\R)$ satisfy
\begin{align*}
\zeta(t) =0 \ \text{ for } t\leq -1 \quad &\text{ and } \quad \zeta (t)=1 \ \text{ for } t\geq 0\\
\eta(t) = 1 \ \text{ for } |t|\le 1/2 \quad &\text{ and } \quad \eta(t)=0 \ \text{ for } |t|\ge 1.
\end{align*}
Since $u=0$ for $x_n<0$, it follows that $\textrm{supp} \ w\subset \overline{(W_-\cup W_+)\cap B_{2\delta}}$. Also $\textrm{supp}[ \mathcal{L}_{A,q}, \chi]u\subset \overline{W_-\cap B_{2\delta}}$ since $[ \mathcal{L}_{A,q}, \chi]u$ involves the derivatives of $\chi$ (i.e. $\partial^\alpha \chi$ where $\alpha$ is a multi-index) and they are zero on $\{|x|\le \delta\}\cup\{\varphi_0> 0\} $ as $\chi = 1 \text{ on } W_+$. 
Now by applying \eqref{carest} with this $w$, along with the fact that $\varphi_0\arrowvert_{W_+}> 0$ and $\varphi_0\arrowvert _{W_-}< 0$, we get
\begin{align*}
\|u\|_{L^2(W_+\cap B_{2\delta})} & \le \|e^{\frac{\varphi_0}{h}}u\|_{L^2(W_+\cap B_{2\delta})}\\
& \le \|e^{\frac{\varphi_0}{h}}\chi u\|_{L^2(B_{4\delta})} \\
& \le C h^{3}\|e^{\frac{\varphi_0}{h}} \mathcal{L}_{A,q} (\chi u)\|_{L^2(B_{4\delta})}\\
& \le C h^{3}\left( \|e^{\frac{\varphi_0}{h}} \chi \mathcal{L}_{A,q} u\|_{L^2(B_{4\delta})} + \|e^{\frac{\varphi_0}{h}}[ \mathcal{L}_{A,q}, \chi] u\|_{L^2(B_{4\delta})} \right) \\
& \le C h^{3} \|e^{\frac{\varphi_0}{h}} [ \mathcal{L}_{A,q}, \chi] u\|_{L^2(W_-\cap B_{2\delta})} \\
& \le C h^{3} \| [ \mathcal{L}_{A,q}, \chi] u\|_{L^2(W_-\cap B_{2\delta})}.
\end{align*}
In the above inequalities, we used the fact that $u$ is a solution of (\ref{3}) and the support conditions. Since $[\mathcal{L}_{A,q},\chi]u$ is a fixed function, letting $h\to 0$ shows that $\|u\|_{L^2(W_+)} =0$.
This proves the proposition in the special case $S = \{x_n = 0\}$.

Next we consider the case where $S$ is a general $C^\infty$ hypersurface. Normalizing, we may assume that $x_0 = 0$ and $S \cap V = \varphi^{-1}_0(0)\cap V$ where $\varphi_0 \in C^\infty (\mathbb{R}^n)$ satisfies $\nabla \varphi_0 \neq 0$ on $S\cap \overline{V}$. After a rotation and scaling, we may also assume $\nabla \varphi_0 (0) = \pm e_n$. We may further assume that $V = B_{4\delta}$ for some $\delta>0$ which can be chosen suitably small but fixed. Taylor approximation near the point $x_0=0$ gives that $\varphi_0(x) = \pm x_n + b(x)|x|^2$ where $|b(x)|\le C$ in $B_{4\delta}$. Thus $S$ looks approximately like $\{x_n =0\}$ in $V$ if $\delta$ is chosen small enough and the two sides of $S$ are given by $V_{\pm} = V \cap \{\pm\varphi_0>0\}$. After these normalizations, we set $\widetilde{\varphi}_0(x) = \varphi_0(x)+ \widetilde{C}\,|x^\prime|^2 \mp \delta^2$, where $\widetilde{C}>0$ will be chosen in order to have $\partial^2_{x_jx_k}\widetilde{\varphi}_0 \geq 0$. Then we can continue the argument given for the above case, replacing $\{\pm x_n >0\}$ by $\{\pm \widetilde{\varphi}_0(x)> 0\}$. This finishes the proof.
\hfill
\end{proof}

Now we prove the Lemma \ref{L2}. Let $\Omega$ be a bounded open set in $\mathbb{R}^n$ and $\widetilde{\Omega}\subset\mathbb{R}^n$ be an another open set such that $\Omega\Subset \widetilde{\Omega}$. Here in this case, our weight function is $\varphi = \pm x_n + |x^\prime|^2 \mp c^2$ near $\widetilde{\Omega}$. 
All we need to check whether the hypothesis \eqref{new5} is satisfied or not, i.e. whether $\{a,b\}\geq 0$ whenever $a=b=0$. Then Lemma \ref{L2} will follow from the Proposition \ref{P2}.  

We find
\begin{equation*}
\partial_{x_j}\varphi =
\begin{cases}
\quad 2x_j \ &\text{ if } \ j\ne n\\
\pm 1 \ &\text{ if } \ j=n
\end{cases}
\end{equation*}
and
\begin{equation*}
\partial^2_{x_jx_k}\varphi =
\begin{cases}
2\delta_{jk} \ &\text{ if } \ j\ne n\\
\ 0 \ &\text{ if } \ j=n.
\end{cases}
\end{equation*}
Correspondingly, the symbols $a(x,\xi)$, $b(x,\xi)$ becomes,
\begin{align*}
a(x,\xi) &= \sum_{j=1}^{n-1}\left( \xi_j^4 + 16 x_j^4 \right) + \left( \xi_n^4 + 1\right) -6\sum_{j=1}^{n-1}\left( 4x_j^2\xi_j^2 \right) +6 \xi_n^2 \\
b(x,\xi) &= 2\sum_{j=1}^{n-1} x_j\xi_j^3 \pm \xi_n^3 - 8\sum_{j=1}^{n-1}x_j^3\xi_j \pm \xi_n. 
\end{align*}
Next we calculate the Poisson bracket $\{a,b\}(x,\xi)$ (cf. \eqref{pb}) to find
\begin{align*} 
\{a,b\}(x,\xi)= & 32 \sum_{j=1}^{n-1} \left[ \xi_j^6 + 3 \xi_j^4 (2x_j)^2+ 3 \xi_j^2 (2x_j)^4+ (2x_j)^6\right] \\
= & 32 \sum_{j=1}^{n-1} \left( \xi_j^2 + 4x_j^2\right)^3\\
\geq&\, 0.
\end{align*}
This completes the discussion of the proof of UCP across the hypersurface.

\section{Weak UCP and UCP for Cauchy data}
\label{S3}
\setcounter{equation}{0}
In this section, we discuss about the proof of the weak UCP (Theorem \ref{T1}) and UCP for the Cauchy data (Theorem \ref{T2}). We first deduce the following proposition which is a special case of weak UCP, from the UCP across a hypersurface. Then
Theorem \ref{T1} follows using a connectedness argument.

\begin{proposition}[(Weak UCP for concentric balls)]
\label{P3}
Let $u\in H^4(B(x_0,R_0))$ satisfies
\begin{equation*}
\begin{aligned}
\mathcal{L}_{A,q}\,u &= 0 \quad \text{ in } B(x_0,R_0)\\
u &= 0 \quad \text{ in } B(x_0,r_0) \ \text{ for some } r_0<R_0.
\end{aligned}
\end{equation*}
Then $u=0$ in $B(x_0,R_0)$.
\end{proposition}

\begin{proof}
Let
\begin{equation*}
I:=\{r\in (0,R_0): u= 0 \text{ in } B(x_0,r)\}.
\end{equation*}
Be the hypothesis, $I$ is a non-empty set. Also it is closed since $u=0$ in $B(x_0,r_i)$ with $r_i\to r$ implies $u=0$ in $B(x_0,r)$. Now we show that $A$ is open as well. Therefore $I=(0,R_0)$ which shows that $u=0$ in $B(x_0,R_0)$, as claimed.

Suppose that $r_1\in I$. Let us consider the hypersurface $S= \partial B(x_0,r_1)$. Since $u=0$ on one side of the hypersurface, for every point $y\in S$, Proposition  \ref{T} says that $u=0$ in some open ball $B(y,\varepsilon_y)\subset B(x_0,R_0)$. Consider the open set
\begin{equation*}
U :=B(x_0,r_1)\cup \left( \underset{y\in S}{\cup} B(y,\varepsilon_y)\right) .
\end{equation*}
As the distance between the compact set $S$ and $\overline{B(x_0,R_0)}\setminus U$ is positive, there exists $\varepsilon>0$ such that $u$ vanishes on $B(x_0,r_1+\varepsilon)$. This implies $I$ is an open set which concludes the proof.
\hfill
\end{proof}

\begin{proof}[Proof of Theorem \ref{T1}]
Let us consider the following set
\begin{equation*}
A:= \{x\in \Omega: u=0 \text{ in } B(x,r) \text{ for some } r>0\}.
\end{equation*}
By the assumption of the theorem, $A$ is non-empty and most importantly it is an open set by its definition. We show that it is also closed as a subset of $\Omega$. Since $\Omega$ is a connected set, this yields that $A = \Omega$ which then completes the proof.

Suppose on the contrary, $A$ is not closed. Let $x$ be a limit point of $A$ such that $x\notin A$, i.e. $u$ does not vanish on $B(x,r)$ for any $r>0$. Let us fix $r$ such that $B(x,r)\subset \Omega$ and let $y\in B(x,r/2)\cap A$, therefore $u=0$ on $B(y,r_0)$ for some $r_0<r/2$. Then Proposition \ref{P3} gives that $u$ vanishes on the concentric ball $B(y,r)$. But this is a contradiction since $x\in B(y, r)$.
\hfill
\end{proof}

Finally we show the unique continuation if the Cauchy data vanishes on some part of the boundary. The proof is done by extending the domain little bit where the Cauchy data vanishes and then applying the weak UCP.

\begin{proof}[Proof of Theorem \ref{T2}]
Let $x_0\in \Gamma$. Since $\Omega$ has smooth boundary, we can assume, upon relabelling the coordinate axes, that  
\begin{equation}
\Omega \cap B(x_0,r) = \{ x\in B(x_0,r): x_n >g(x')\}
\end{equation}
for some $r>0$ and some $g:\R^{n-1}\to \R$ a $C^\infty$-function. Now we would like to extend the domain near $x_0$. Let $h\in C_c^\infty(\R^{n-1})$ be a function such that $h(x') =0$ if $|x'|\ge r/2 $ and $h(x') =1$ if $|x'|\le r/4 $. We define the set, for $\varepsilon>0$,
\begin{equation*}
\widetilde{\Omega}:=\Omega \cup \{x\in B(x_0,r): x_n >g(x') - \varepsilon h(x')\}.
\end{equation*}
If $\varepsilon$ is small enough, $\{x: |x'|\le r/2, x_n =g(x') - \varepsilon h(x')\} \subset B(x_0,r)$. Clearly $\widetilde{\Omega}$ is an open, bounded, connected set with smooth boundary. Also we define $u$ on the extended domain as
\begin{equation*}
\widetilde{u}(x) :=
\begin{cases}
u(x), \quad \text{ if } x\in\Omega\\
0, \quad \text{ if } x\in \widetilde{\Omega}\setminus \overline{\Omega}.
\end{cases}
\end{equation*}
Since $\widetilde{u}\arrowvert_\Omega \in H^4(\Omega)$ and $\widetilde{u}\arrowvert_{\widetilde{\Omega}\setminus \overline{\Omega}}\in H^4(\widetilde{\Omega}\setminus \overline{\Omega})$, we may conclude $\widetilde{u}\in H^4(\widetilde{\Omega})$ if the traces match at the interface from both sides. But from the hypothesis, $u= \partial_\nu u = \partial^2_\nu u = \partial^3_\nu u =0$ on $\Gamma$. Also note that by the construction, $\partial\widetilde{\Omega} \setminus \partial \Omega\subset \Gamma$. Therefore, we obtain $\widetilde{u}\in H^4(\widetilde{\Omega})$. Furthermore, extending $q$ by $0$ in $\widetilde{\Omega}\setminus \overline{\Omega}$, we get $\widetilde{q}\in L^\infty(\widetilde{\Omega})$. Similarly, consider $\widetilde{A}\in W^{1,\infty}(\widetilde{\Omega})$, an extension of $A$. Then it follows
\begin{equation*}
\mathcal{L}_{\widetilde{A}, \widetilde{q}}\widetilde{u} =0 \quad \text{a.e. in } \widetilde{\Omega}.
\end{equation*}
Now since $\widetilde{u} =0$ in $\widetilde{\Omega}\setminus \overline{\Omega}$, the weak UCP (Theorem \ref{T1}) yields that $\widetilde{u}$ vanishes on the whole domain $\widetilde{\Omega}$ . Hence, $\widetilde{u}\arrowvert_\Omega =u=0$ which proves the theorem.
\hfill
\end{proof}

\section{Stability estimate}
\label{S4}
\setcounter{equation}{0}
Here we apply the Carleman estimates to establish the corresponding stability estimate. 
In order to do so, some Caccioppoli-type interior estimate for the fourth order operator is also crucial which we prove below.

\begin{proposition}[(Caccioppoli inequality)]
Let $\mathcal{L}_{A,q}u =0$ in $B_1$. For fixed $r,\rho\in (0,1)$ with $r>\rho$, there exists a constant $C>0$ depending only on $\|A\|_{W^{1,\infty}(B_1)}$ and $\|q\|_{L^\infty(B_1)}$ such that
\begin{equation}
\label{caccioppoli}
\int\displaylimits_{B_r\setminus\overline{B}_{\varrho}}{(|D^2u|^2+|D^3u|^2)}\le \frac{C}{(r-\varrho)^2}\int\displaylimits_{B_{2r}\setminus\overline{B}_{\frac{\varrho}{2}}}{(|u|^2+|Du|^2)} .
\end{equation}
\end{proposition}

\begin{proof}
We start with estimating the first term in the left hand side of (\ref{caccioppoli}). From the equation satisfied by $u$, we get, for any $\psi\in C_c^2(B_1)$,
\begin{equation}
\label{6}
\begin{aligned}
0 = \int\displaylimits_{B_1}{\mathcal{L}_{A,q}\,u \,\psi} =
\int\displaylimits_{B_1}{D_{x_j}^2 u\, D_{x_j}^2 \psi} + \int\displaylimits_{B_1} {A_j\, D_{x_j}u\, \psi} + \int\displaylimits_{B_1}{qu\, \psi}.
\end{aligned}
\end{equation}
Choose a cut-off function $\widetilde{\eta}\in C_c^\infty(B_1)$ which satisfies 
\begin{equation*}
\begin{aligned}
 0\le \widetilde{\eta} \le 1 \quad &\text{ in } \ B_1,\\
 \widetilde{\eta} =1 \quad &\text{ in } \ B_r\setminus \overline{B}_{\varrho},\\
\widetilde{\eta}=0 \quad &\text{ outside } \ B_{2r}\setminus\overline{B}_{\frac{\varrho}{2}},\\
\text{ and } \ |D^k\widetilde{\eta}| \le \frac{c}{(r-\varrho)^k}\quad &\text{ for } \ k=1,...,4.
\end{aligned}
\end{equation*}
Substituting the test function $\psi$ by $\widetilde{\eta}^4 u$ in (\ref{6}) yields,
\begin{equation*}
\begin{aligned}
0 & = \int\displaylimits_{B_1}{D_{x_j}^2 u\, D_{x_j}^2 (\widetilde{\eta}^4 u)} + \int\displaylimits_{B_1} {A_j D_{x_j}u\, (\widetilde{\eta}^4 u)} + \int\displaylimits_{B_1}{q\widetilde{\eta}^4 |u|^2}\\
& = \int\displaylimits_{B_1}{\widetilde{\eta}^4 |D_{x_j}^2 u|^2} +2\int\displaylimits_{B_1}{D_{x_j}^2u\, D_{x_j}\widetilde{\eta}^4 D_{x_j}u } +\int\displaylimits_{B_1}{u\, D_{x_j}^2 u\, D_{x_j}^2\widetilde{\eta}^4}  + \int\displaylimits_{B_1} {A_j\widetilde{\eta}^4 \, u\, D_{x_j}u} + \int\displaylimits_{B_1}{q\widetilde{\eta}^4 |u|^2}
\end{aligned}
\end{equation*}
which can be re-written as, employing Young's inequality,
\begin{equation*}
\begin{aligned}
& \qquad \int\displaylimits_{B_1}{\widetilde{\eta}^4 |D_{x_j}^2 u|^2}\\
& \le C \left( \, \int\displaylimits_{B_1}{\widetilde{\eta}^3 |D_{x_j}^2u| |D_{x_j}\widetilde{\eta}| |D_{x_j}u| } +\int\displaylimits_{B_1}{|u| |D_{x_j}^2 u| \left( \widetilde{\eta}^3|D_{x_j}^2\widetilde{\eta}|+ \widetilde{\eta}^2 |D_{x_j}\widetilde{\eta}|^2\right) }\right.\\
& \hspace{8cm}\left. + \int\displaylimits_{B_1} {\widetilde{\eta}^4 |u| |D_{x_j}u|}+ \int\displaylimits_{B_1}{\widetilde{\eta}^4 |u|^2}\right)\\
& \le \varepsilon\int\displaylimits_{B_1}{\widetilde{\eta}^4 |D_{x_j}^2u|^2}+ C_\varepsilon \int\displaylimits_{B_1} {|D_{x_j}\widetilde{\eta}|^2 |D_{x_j}u|^2} + C_\varepsilon \int\displaylimits_{B_1} {|D^2_{x_j}\widetilde{\eta}|^2 |u|^2} + C_\varepsilon \int\displaylimits_{B_1}{|D_{x_j}\widetilde{\eta}|^4 |u|^2}\\
& \hspace{7.5cm} + C \int\displaylimits_{B_1}{\widetilde{\eta}^2|u|^2} + C \int\displaylimits_{B_1} {\widetilde{\eta}^2|D_{x_j}u|^2}
\end{aligned}
\end{equation*}
where the above constant $C>0$ depends on $\|A\|_{W^{1,\infty}(B_1)}$, $\|q\|_{L^\infty(B_1)}$ only.
Next incorporating the properties of $\widetilde{\eta}$ and choosing $\varepsilon$ suitably to absorb the first term of the right hand side in the left hand side, we obtain,
\begin{align}
& \quad \int\displaylimits_{B_r\setminus\overline{B}_{\varrho}}{|D_{x_j}^2 u|^2}\nonumber\\
& \le \int\displaylimits_{B_1}{\widetilde{\eta}^4 |D_{x_j}^2 u|^2}\nonumber\\
& \le C \left( \, \int\displaylimits_{B_1} {|D_{x_j}\widetilde{\eta}|^2 |D_{x_j}u|^2} + \int\displaylimits_{B_1} {|D^2_{x_j}\widetilde{\eta}|^2 |u|^2} + \int\displaylimits_{B_1}{|D_{x_j}\widetilde{\eta}|^4 |u|^2} + \int\displaylimits_{B_1}{\widetilde{\eta}^2(|u|^2+|D_{x_j}u|^2)}\right)\nonumber\\
& \le \frac{C}{(r-\varrho)^2}\int\displaylimits_{B_{2r}\setminus\overline{B}_{\frac{\varrho}{2}}} {|D_{x_j}u|^2} + \frac{C}{(r-\varrho)^4}\int\displaylimits_{B_{2r}\setminus\overline{B}_{\frac{\varrho}{2}}} {|u|^2} + C \int\displaylimits_{B_{2r}\setminus\overline{B}_{\frac{\varrho}{2}}}{(|u|^2+|D_{x_j}u|^2)}. \label{7}
\end{align}
This completes the estimate involving the second order term $D^2u$.

Similarly to estimate the term $D^3u$ in terms of $Du$ and $u$, we repeat the above arguments with the test function $\widetilde{\eta}^2D^2_{x_j}u$,
\begin{equation*}
\begin{aligned}
0 & =-\int\displaylimits_{B_1}{D_{x_j}^3 u\, D_{x_j} (\widetilde{\eta}^2D^2_{x_j}u)} + \int\displaylimits_{B_1} {A_j\, D_{x_j}u\, (\widetilde{\eta}^2D^2_{x_j}u)} + \int\displaylimits_{B_1}{q\, u (\widetilde{\eta}^2D^2_{x_j}u)}\\
& =-\int\displaylimits_{B_1}{\widetilde{\eta}^2\, |D_{x_j}^3 u|^2} -\int\displaylimits_{B_1}{D_{x_j}^3 u\, D_{x_j} \widetilde{\eta}^2\, D^2_{x_j}u}+ \int\displaylimits_{B_1} {A_j \widetilde{\eta}^2\, D_{x_j}u\, D^2_{x_j}u} + \int\displaylimits_{B_1}{q \widetilde{\eta}^2\, u\, D^2_{x_j}u}
\end{aligned}
\end{equation*}
which implies
\begin{equation*}
\begin{aligned}
\int\displaylimits_{B_1}{\widetilde{\eta}^2\, |D_{x_j}^3 u|^2} &\le C \left( \, \int\displaylimits_{B_1}{\widetilde{\eta}|D_{x_j}^3 u| |D_{x_j} \widetilde{\eta}| |D^2_{x_j}u|}+ \int\displaylimits_{B_1} {\widetilde{\eta}^2 |D_{x_j}u| |D^2_{x_j}u|} + \int\displaylimits_{B_1}{\widetilde{\eta}^2|u| |D^2_{x_j}u|}\right)\\
& \le \varepsilon \int\displaylimits_{B_1}{\widetilde{\eta}^2\, |D_{x_j}^3 u|^2} + C_\varepsilon \int\displaylimits_{B_1}{|D_{x_j} \widetilde{\eta}|^2 |D^2_{x_j}u|^2}+ C \int\displaylimits_{B_1}{\widetilde{\eta}^2 \left( |u|^2 + |D_{x_j}u|^2 + |D^2_{x_j}u|^2\right) } .
\end{aligned}
\end{equation*}
Therefore,
\begin{align}
\int\displaylimits_{B_r\setminus\overline{B}_{\varrho}}{|D_{x_j}^3 u|^2} &\le \frac{C}{(r-\varrho)^2}\int\displaylimits_{B_{2r}\setminus\overline{B}_{\frac{\varrho}{2}}}{|D^2_{x_j}u|^2}+ C \int\displaylimits_{B_{2r}\setminus\overline{B}_{\frac{\varrho}{2}}}{\left( |u|^2 + |D_{x_j}u|^2 + |D^2_{x_j}u|^2\right) }\nonumber\\
&\le \frac{C}{(r-\varrho)^2} \left( \int\displaylimits_{B_{4r}\setminus\overline{B}_{\frac{\varrho}{4}}} {(|D_{x_j}u|^2+|u|^2)} \right) \label{8}
\end{align}
where the above constant $C>0$ depends on $\|A\|_{W^{1,\infty}(B_1)}$, $\|q\|_{L^\infty(B_1)}$ only. Thus, (\ref{7}) together with (\ref{8}) completes the proof. 
\hfill
\end{proof}

Next we establish the stability estimate. For this, it is more interesting to work with boundary value problems.
Recall that for any $\varphi$ smooth function,
\begin{equation*}
\Omega_\delta := \Omega \cap \{\varphi >\delta\} \quad \text{ and } \quad \partial\Omega_\delta := \partial\Omega\cap \{\varphi >\delta\}.
\end{equation*}

\begin{proof}[Proof of Theorem \ref{T4}]
We use here the analogue of the Carleman estimate (\ref{c-e}) for boundary value problems. By lifting the trace operator, there exists $\upsilon\in H^4(\Omega)$ satisfying
\begin{equation*}
\partial^k_{\nu} \upsilon = g^k \quad \text{ on } \Gamma, \ k=0,...,3
\end{equation*}
with
\begin{equation}
\label{13}
\|\upsilon\|_{H^4(\Omega)}\le C \sum_{k=0}^{3}\|g^k\|_{H^{\frac{7}{2}-k}(\Gamma)}
\end{equation}
for some constant $C>0$ depending only on $\Omega$ and $\Gamma$. Setting $u^* = u-\upsilon$, $u^*$ satisfies the following Cauchy problem
\begin{equation*}
\begin{cases}
\mathcal{L}_{A,q}\, u^* &= f + \mathcal{L}_{A,q}\, \upsilon \quad \text{ in } \ \Omega,\\
\partial^k_{\nu} u^* &= 0 \qquad \qquad \ \ \text{ on } \ \Gamma, \quad k=0,...,3.
\end{cases}
\end{equation*}
Now the Carleman estimate says that (cf. Proposition \ref{P2}) there exists $C>0$, depending on only $\|A\|_{W^{1,\infty}(\Omega)}$, $\|q\|_{L^\infty(\Omega)}$, $\Omega$, $n$,
such that, for all $w\in C_c^\infty(\Omega)$ and $0<h<1$ small enough,
\begin{equation}
\label{12}
\int\displaylimits_{\Omega}{\left( |w|^2 + h^2 |\nabla w|^2\right) e^{\frac{2\, \varphi}{h}}} \leq C h^6  \int\displaylimits_{\Omega}{|\mathcal{L}_{A,q}w|^2 e^{\frac{2\, \varphi}{h}}}
\end{equation}
where $\varphi$ is any Carleman weight. Let us introduce a cut-off function $\eta\in C^\infty_c(\Omega)$ such that $0\le \eta \le 1$ in $\Omega$, $\eta = 1$ in $\Omega_{\delta/2}$ and $\eta =0$ outside $\Omega_0$.
Since $\eta u^*\in H^4_0(\Omega)$, we may apply the Carleman estimate (\ref{12}) with $w= \eta u^*$ to obtain
\begin{equation*}
\int\displaylimits_{\Omega_0}{\left( |\eta u^*|^2 + h^2 |\nabla (\eta u^*)|^2\right) e^{\frac{2\, \varphi}{h}}} \leq C h^6  \int\displaylimits_{\Omega_0}{|\mathcal{L}_{A,q}(\eta u^*)|^2 e^{\frac{2\, \varphi}{h}}}.
\end{equation*}
Since $\eta =1$ in $\Omega_{\delta/2}$, we can further bound the left hand side from below as,
\begin{equation*}
\int\displaylimits_{\Omega_{\delta/2}}{\left( | u^*|^2 + h^2 |\nabla u^*|^2\right) e^{\frac{2\, \varphi}{h}}} \leq C h^6  \int\displaylimits_{\Omega_0}{|\mathcal{L}_{A,q}(\eta u^*)|^2 e^{\frac{2\, \varphi}{h}}}.
\end{equation*}
Now we calculate the right hand side,
\begin{equation*}
\mathcal{L}_{A,q}(\eta u^*) = \eta \mathcal{L}_{A,q}\,u^* + [\mathcal{L}_{A,q}, \eta]u^* = \eta \left( f + \mathcal{L}_{A,q}\, \upsilon\right)  + [\mathcal{L}_{A,q}, \eta]u^* 
\end{equation*}
where
\begin{equation*}
[\mathcal{L}_{A,q}, \eta]u^* = u\, D_{x_j}^4 \eta + 4 D_{x_j}^3\eta\, D_{x_j}u^* + 6 D_{x_j}^2 \eta\, D_{x_j}^2 u^* + 4 D_{x_j}\eta\, D_{x_j}^3 u^*.
\end{equation*}
Also,
\begin{equation*}
|[\mathcal{L}_{A,q}, \eta]u^*| \le C \left( |u^*| + |Du^*| + | D^2u^*| + |D^3u^*|\right)
\end{equation*}
for some constant $C>0$ which depends only on $\delta$.
Therefore, taking into account that $[\mathcal{L}_{A,q}, \eta]u^*$ consists of the derivatives of $\eta$, thus $\textrm{supp} [\mathcal{L}_{A,q}, \eta]u^*\subset \Omega_0\setminus \overline{\Omega}_{\delta/2} $, the Carleman estimate becomes, 
\begin{equation*}
\begin{aligned}
& \qquad \int\displaylimits_{\Omega_{\delta/2}}{\left( |u^*|^2 + h^2 |\nabla u^*|^2\right) e^{\frac{2\, \varphi}{h}}}\\
& \le C\, h^6 \int\displaylimits_{\Omega_0 } {\left( |f|^2 + |\mathcal{L}_{A,q}\, \upsilon|^2\right) e^{\frac{2\, \varphi}{h}}} +C\, h^6 \int\displaylimits_{\Omega_0\setminus \overline{\Omega}_{\delta/2} }{\left( |u^*|^2 + |Du^*|^2 + |D^2u^*|^2 + |D^3u^*|^2\right)e^{\frac{2\, \varphi}{h}}}
\end{aligned}
\end{equation*}
which reduces to, using the fact that $0<h<1$ and denoting by $\Phi := \underset{\overline{\Omega}}{\sup}\, \varphi$ and using $\Omega_{\delta}\subset \Omega_{\delta/2}$,
\begin{equation*}
\begin{aligned}
& \quad \int\displaylimits_{\Omega_{\delta/2}}{\left( |u^*|^2 + |\nabla u^*|^2\right) e^{\frac{2\, \varphi}{h}}}\\
&\leq C e^{\frac{2}{h}\Phi} \int\displaylimits_{\Omega_0}{\left( |f|^2 + |\mathcal{L}_{A,q}\upsilon|^2\right) } + C e^{\frac{\delta}{h}} \int\displaylimits_{\Omega_0 \setminus \overline{\Omega}_{\delta/2}} {\left( |u^*|^2 + |Du^*|^2 + |D^2u^*|^2 + |D^3u^*|^2\right)}.
\end{aligned}
\end{equation*}
Further plugging the Caccioppoli estimate (\ref{caccioppoli}) in the right hand side and replacing the left hand side on smaller domain, we get,
\begin{equation*}
\begin{aligned}
e^{\frac{2}{h}\delta}\int\displaylimits_{\Omega_{\delta}}{\left( |u^*|^2 + |\nabla u^*|^2\right) }&\leq \int\displaylimits_{\Omega_{\delta/2}}{\left( |u^*|^2 + |\nabla u^*|^2\right) e^{\frac{2\, \varphi}{h}}}\\
&\leq C e^{\frac{2}{h}\Phi} \int\displaylimits_{\Omega_0}{\left( |f|^2 + |\mathcal{L}_{A,q}\upsilon|^2\right) } + C e^{\frac{\delta}{h}} \int\displaylimits_{\Omega_0 \setminus \overline{\Omega}_{\delta/2}} {\left( |u^*|^2 + |\nabla u^*|^2 \right)}
\end{aligned}
\end{equation*}
for any $0<h<1$ suitably small, say for $h<h_1$.
Also the above constant $C>0$ depends only on $\Omega, \|A\|_{W^{1,\infty}(\Omega)}$, $\|q\|_{L^\infty(\Omega)}, n$ and $\delta$. Simplifying the above estimate, along with the estimate (\ref{13}), we get
\begin{align}
\|u^*\|_{H^1(\Omega_\delta)} & \le C \left( e^{\frac{1}{h}(\Phi - \delta)} F + e^{-\frac{\delta}{2h}} \|u^*\|_{H^1(\Omega_0)}\right)\nonumber\\
& \le C \left( e^{\frac{1}{h}(\Phi - \delta)} F + e^{-\frac{\delta}{2h}} \|u\|_{H^1(\Omega_0)}+ e^{-\frac{\delta}{2h}} \|\upsilon\|_{H^1(\Omega_0)}\right)\nonumber\\
& \le C \left( e^{\frac{1}{h}(\Phi - \delta)} F + e^{-\frac{\delta}{2h}} M + F\right). \label{17}
\end{align}

Now if $M<F$, then trivially we can write,
\begin{equation*}
M =M^{1-\theta} M^\theta \le M^{1-\theta} F^\theta \quad \text{ for any } \theta\in (0,1)
\end{equation*}
which implies
\begin{equation*}
\|u\|_{H^1(\Omega_\delta)} \le \|u\|_{H^1(\Omega_0)} = M \le M^{1-\theta} F^\theta.
\end{equation*}

If $M\ge F$, we choose
\begin{equation}
\label{18}
\frac{1}{h_0}:= \frac{1}{\Phi - \frac{\delta}{2}} \ln \left( \frac{M}{F}\right) \qquad \text{i.e.} \qquad e^{\frac{1}{h_0}(\Phi - \delta)} F = e^{-\frac{\delta}{2h_0}} M.
\end{equation}
Assume that $\Phi>\delta$ (otherwise the estimate (\ref{SE}) holds trivially being $\Omega_{\delta} = \emptyset$), hence $h_0>0$.
Further we consider two cases:
 
\textbf{(i)} Let $h_0\le h_1$. Then we choose $h=h_0$ in (\ref{17}) to get, with the help of (\ref{18}),
\begin{equation*}
\|u^*\|_{H^1(\Omega_{\delta})}\le 2C e^{-\frac{\delta}{2h_0}}M + C F .
\end{equation*}
But
\begin{equation*}
e^{-\frac{\delta}{2h_0}} = \left( \frac{F}{M}\right) ^\theta \quad \text{ where } \ \theta = \frac{\delta}{2\Phi - \delta}
\end{equation*}
which implies,
\begin{equation*}
\|u^*\|_{H^1(\Omega_{\delta})}\le C M^{1-\theta}F^\theta + CF.
\end{equation*}
Note that $0<\theta <1$ as well. This finally gives, plugging in $u^* = u-\upsilon$ and the estimate (\ref{13}),
\begin{equation*}
\|u\|_{H^1(\Omega_{\delta})}\le C (F+M^{1-\theta}F^\theta).
\end{equation*}

\textbf{(ii)} Let $h_0 > h_1$. From (\ref{18}), it follows,
$M \le e^{\frac{1}{h_1}(\Phi - \frac{\delta}{2})} F$
which yields
\begin{equation*}
\|u\|_{H^1(\Omega_{\delta})} \le M = M^{1-\theta}M^\theta \le e^{\frac{\delta}{2h_1}} M^{1-\theta} F^\theta.
\end{equation*}
This completes the proof.
\hfill\end{proof}


\end{document}